\documentclass[11pt]{article}

\usepackage{epsfig}
\usepackage{graphicx}
\usepackage{amsbsy}
\usepackage{amsmath}
\usepackage{amsfonts}
\usepackage{amssymb}
\usepackage{textcomp}
\usepackage{hyperref}
\usepackage{aliascnt}

\newcommand{\mcm}[3]{\newcommand{#1}[#2]{{\ensuremath{#3}}}} 

\mcm{\tuple}{1}{\langle #1 \rangle}
\mcm{\name}{1}{\ulcorner #1 \urcorner}
\mcm{\Nbb}{0}{\mathbb{N}}
\mcm{\Zbb}{0}{\mathbb{Z}}
\mcm{\Rbb}{0}{\mathbb{R}}
\mcm{\Cbb}{0}{\mathbb{C}}
\mcm{\Qbb}{0}{\mathbb{Q}}
\mcm{\Bcal}{0}{\cal B}
\mcm{\Ccal}{0}{\cal C}
\mcm{\Dcal}{0}{\cal D}
\mcm{\Ecal}{0}{\cal E}
\mcm{\Fcal}{0}{\cal F}
\mcm{\Gcal}{0}{\cal G}
\mcm{\Hcal}{0}{\cal H}
\mcm{\Ical}{0}{\cal I}
\mcm{\Jcal}{0}{\cal J}
\mcm{\Kcal}{0}{\cal K}
\mcm{\Lcal}{0}{\cal L}
\mcm{\Mcal}{0}{\cal M}
\mcm{\Ncal}{0}{\cal N}
\mcm{\Ocal}{0}{{\cal O}}
\mcm{\Pcal}{0}{{\cal P}}
\mcm{\Qcal}{0}{{\cal Q}}
\mcm{\Rcal}{0}{{\cal R}}
\mcm{\Scal}{0}{{\cal S}}
\mcm{\Tcal}{0}{{\cal T}}
\mcm{\Ucal}{0}{{\cal U}}
\mcm{\Vcal}{0}{{\cal V}}
\mcm{\Xcal}{0}{{\cal X}}
\mcm{\Ycal}{0}{{\cal Y}}
\mcm{\Mfrak}{0}{\mathfrak M}

\mcm{\restric}{0}{\upharpoonright}
\mcm{\upset}{0}{\uparrow}
\mcm{\onto}{0}{\twoheadrightarrow}
\mcm{\smallNbb}{0}{{\small \mathbb{N}}}
\DeclareMathOperator{\preop}{op}
\mcm{\op}{0}{^{\preop}}

\newcommand{\supst}{\textsuperscript{st}}

\newcommand{\se}{\subseteq}
{\begin{array}{c}
\setlength{\unitlength}{1em}}%
{\end{array}}

\usepackage{amsthm}

\newcommand{\theoremize}[2]{\newaliascnt{#1}{thm} \newtheorem{#1}[#1]{#2} \aliascntresetthe{#1}}

\theoremstyle{plain}
\newtheorem{thm}{Theorem}[section]
\theoremize{lem}{Lemma}
\theoremize{sublem}{Sublemma}
\theoremize{claim}{Claim}
\theoremize{rem}{Remark}
\theoremize{prop}{Proposition}
\theoremize{cor}{Corollary}
\theoremize{que}{Question}
\theoremize{oque}{Open Question}
\theoremize{con}{Conjecture}

\theoremstyle{definition}
\theoremize{dfn}{Definition}
\theoremize{eg}{Example}
\theoremize{exercise}{Exercise}
\theoremstyle{plain}

\usepackage{verbatim}
\usepackage{enumerate}
\usepackage{bbold}
\usepackage[all]{xy}

\usepackage{subfig}

\newcommand{\sm}{\setminus}

\DeclareMathOperator{\supp}{supp}

\newcommand{\ct}{^\complement}

\begin{document}
\title{Infinite trees of matroids}
\author{Nathan Bowler and Johannes Carmesin}
\maketitle

\begin{abstract}
We generalise the construction of infinite matroids from trees of matroids to allow the matroids at the nodes, as well as the field over which they are represented, to be infinite.
\end{abstract}

\section{Introduction}

In 2008, Bruhn et al \cite{matroid_axioms} introduced several equivalent axiomatisations for infinite matroids, providing a foundation on which a theory of infinite matroids with duality can be built. 
We shall work with a slightly better behaved subclass of infinite matroids, called tame matroids. This class includes all finitary matroids and all the other motivating examples of infinite matroids but  
is easier to work with than the class of infinite matroids in general
\cite{THINSUMS}, \cite{BC:rep_matroids}, \cite{BC:determinacy},   \cite{BC:wild_matroids},  \cite{BC:ubiquity}, \cite{BCC:graphic_matroids}.
In \cite{BC:determinacy}, we gave a construction by means of which 
finite matroids can be stuck together to get infinite 
tame matroids. The construction of \cite{BC:determinacy} was restricted to the countable setting. In this paper, we extend it to the general setting.
%In this paper, we will extend the construction of limits of matroids from %\cite{BC:determinacy} to uncountable settings.

\vspace{0.3cm}

A large collection of motivating examples of infinite matroids arises from locally finite graphs $G$.
First of all, two well-established matroids associated to such a graph $G$ are the  \emph{finite cycle matroid $M_{FC}(G)$}, whose circuits are the finite cycles in $G$, and the \emph{topological cycle matroid $M_{TC}(G)$}, whose circuits are the edge sets of topological circles in the topological space $|G|$ obtained from $G$ by adding the ends \cite{RD:HB:graphmatroids}. More generally, we say a tame matroid is a {\em $G$-matroid} if all of its circuits are edge sets of topological circles in $|G|$ and all of its cocircuits are bonds of $G$. Thus both $M_{FC}(G)$ and $M_{TC}(G)$ are $G$-matroids.

It turns out that any $G$-matroid $M$ is determined by a set $\Psi$ of ends of $G$, in that the circuits of $M$ are the \emph{$\Psi$-circuits}, that is, the edge sets of those topological circles that only use ends from $\Psi$ \cite{BC:ubiquity}.
Unfortunately, there are graphs $G$ and sets $\Psi$ of ends such that the set of $\Psi$-circuits is not the set of circuits of a matroid \cite{BC:determinacy}.
But this can only happen if $\Psi$ is topologically unpleasant.

\begin{thm}[\cite{BC:determinacy}]\label{ref2}
 Let $\Psi$ be a Borel set of ends of a locally finite graph $G$. Then the $\Psi$-circuits of $G$ are the circuits of a matroid. 
\end{thm}

$\Psi$-circuits in graphs are a special case of a more general construction:
A \emph{tree of presentations} $\Tcal$ over a field $k$ consists of a tree $T_\Tcal$ 
with a matroid presentaed over $k$ at each node, where the ground sets of these matroids are only allowed to overlap if they are at adjacent nodes.
Very roughly, the \emph{circuits} of $\Tcal$ are obtained by gluing together local circuits at the nodes of some subtree. Given a set $\Psi$ of ends of $T$, the $\Psi$-circuits are those circuits of $\Tcal$
 for which the underlying subtree has all its ends in $\Psi$.
The following theorem implies \autoref{ref2}.

\begin{thm} [\cite{BC:determinacy}]\label{ref1}
Let $k$ be a finite field and $\Tcal$ be a tree of presentations\footnote{In \cite{BC:determinacy} 
we worked with `trees of matroids' instead of `trees of presentations'. We have made this change since the set of $\Psi$-circuits we get for such a tree may in general not only depend on the structure of the matroids attached to each node but also on their presentations.} 
over $k$.
If $\Psi$ is a Borel set of ends of $T_\Tcal$, then the $\Psi$-circuits are the circuits of a matroid, called the $\Psi$-matroid of $\Tcal$.
\end{thm}

$\Psi$-matroids also appear naturally in the study of planar duality for infinite graphs \cite{DP:dualtrees} and in the reconstruction theorem of tame matroids from their canonical decompositions into 3-connected minors \cite{BC:ubiquity}. 

The $\Psi$-matroid construction of \autoref{ref1} can be thought of as giving, in a sense, a limit of the matroids induced from finite subtrees, but with the advantage that we are able to freely specify a great deal of information `at infinity', namely the set $\Psi$.
If we choose $\Psi$ to be empty, this corresponds to taking the direct limit.
On the other hand, taking $\Psi$ to be the set of all ends corresponds to taking the inverse  limit.

The purpose of this paper is to prove an extension of \autoref{ref1}. We have proved this extension with an application in mind: it is
used as a tool in the proof of an  extension of \autoref{ref2} to arbitrary graphs \cite{C:undom_td}. 
We have tried to avoid the need for further generalisations by making 
the version in this paper as general as possible.
%This in particular gives a construction of the topological cycle matroids for an arbitrary graph, which in general is neither finitary nor cofinitary.

\vspace{0.3 cm}

Next, let us discuss which of the restrictions from \autoref{ref1} we can weaken.
First, in that theorem all of the matroids at the nodes are required to be finite. Allowing arbitrary infinite matroids at the nodes is unfortunately not possible - in fact, even for infinite stars of matroids in which the central node is infinite but all leaves are finite it is possible for our gluing construction to fail to give a matroid. But this is the only problem - that is, we are able to show that if the matroids at the nodes of the tree work well when placed at the centre of such stars, then they can also be glued together along arbitrary trees. The advantage of this approach is its great generality, but the disadvantage is that the class of {\em stellar} matroids, that is, those which fit well at the centre of stars, is not characterised in simpler terms. However, since in all existing applications the matroids involved can be easily seen to be stellar we do not see this as a great problem.

Second, in \autoref{ref1} the matroids at the nodes were required to be representable over a common field $k$. This continues to play a necessary role in our construction, because it is based on a gluing construction for finite matroids which in turn relies on representability over a common field. Because we now allow the matroids at the nodes to be infinite, we require them to be representable in the sense of \cite{THINSUMS}, which introduced a notion of representability for infinitary matroids. 

However, we are able to drop the requirement that $k$ be finite.
Thus we obtain the following more general result.

\begin{thm}\label{main_intro}
Let $k$ be any field and let $\Tcal$ be a stellar tree of presentations presented over $k$, and let $\Psi$ be a Borel set of ends of $\Tcal$. Then the $\Psi$-circuits are the circuits of a matroid.
\end{thm}

We not only extend \autoref{ref2}, we also give a new and simpler proof of it, see \autoref{easy_o2}.

The paper is organised as follows. After recalling some preliminaries in \autoref{prelims}, in \autoref{easy_o2} we give a new proof of \autoref{ref2} which is simpler than the original one. 
However, to understand the rest of this paper, it is not necessary to read that section.
In the proof of our main result, we will rely on the determinacy of certain games, and in \autoref{games} we prove a lemma that allows us to simplify winning strategies in  these games. We then introduce presentations of infinite matroids over a field in \autoref{pres}, and the gluing construction along a tree in \autoref{treepres}. The proof that this construction gives rise to matroids is given in Sections \ref{O2} and \ref{IM}.

\section{Preliminaries}\label{prelims}

Throughout, notation and terminology for (infinite) graphs are those of~\cite{DiestelBook10}, and for matroids those of~\cite{Oxley,matroid_axioms}. We will rely on the following lemma from \cite{DiestelBook10}:

\begin{lem}[K\"onig's Infinity Lemma \cite{DiestelBook10}]\label{Infinity_Lemma}
Let $V_0,V_1,\ldots$ be an infinite sequence of disjoint non-empty finite sets, and let $G$ be a graph on their union. Assume that every vertex $v$ in $V_n$ with $n\geq 1$ has a neighbour $f(v)$ in $V_{n-1}$. Then $G$ includes a ray $v_0v_1\ldots$ with $v_n\in V_n$ for all $n$.
\end{lem}

Note that this is equivalent to the usual formulation of K\"onig's Lemma, namely that every locally finite rayless tree is finite.

Given a graph $G$, the set of its ends is denoted by $\Omega(G)$.
An end $\omega$ is in the \emph{closure of} some edge set $F$ if for each finite separator $S$, the unique component of $G\sm S$ including a tail of each ray in $\omega$ contains a vertex incident with an edge of $F$ in $G$.

A \emph{walk} in a digraph is a sequence $w_1...w_n$ of vertices such that $w_iw_{i+1}$
is an edge for each $i<n$.
For a walk $W=w_1...w_n$ and a vertex $x$ in $W$, 
let $i$ be minimal with $w_i=x$.
Then we denote by $Wx$ the walk $w_1...w_i$ and by $xW$ the walk $w_i...w_n$.

In this paper we define \emph{tree decompositions} slightly differently than in \cite{DiestelBook10}.
Namely, we impose the additional requirement that each edge of the graph is contained in a unique part of the tree decomposition. Clearly, each tree decomposition in the sense of 
\cite{DiestelBook10} can easily be transformed into such a tree decomposition of the same width.
Throughout this paper, \emph{even} means finite and a multiple of 2. 

\vspace{0.3 cm}

Having dealt with the graph theoretic preliminaries, we now define positional games. 
A \emph{positional game} is played in a digraph $D$ with a marked starting vertex $a$. The vertices of the digraph are called {\em positions} of the game.
The game is played between two players between whom play alternates.
At any point in the game, there is a \emph{current position}, which initially is $a$.
In each move, the player whose turn it is to play picks an out-neighbour $x$ of the current position, and then the current position is updated to $x$. Thus a \emph{play} in this game is 
encoded as a walk in $D$ starting at an out-neighbour of $a$. If a player cannot move, they lose. 
 If play continues forever, then the players between them generate an infinite walk starting at a neighbour of $a$. Then the first player wins if this walk is in the \emph{set $\Phi$ of winning conditions}, which is part of the data of the positional game.

A \emph{strategy for the first player} is a set $\sigma$ of finite plays $P$ all ending with a move of the first player
such that the following is true for all $P\in \sigma$:
Let $m$ be a move of the second player such that $P m$ is a legal play.
Then there is a unique move $m'$ of the first player such that $P m m'\in \sigma$.
Furthermore, we require that $\sigma$ is closed under \emph{2-truncation}, that is, 
for every nontrivial $P\in \sigma$ there are some
$P'\in \sigma$ and moves $m$ and $m'$ of the second player and the first player, respectively, such that
$P' m m'=P$.

An infinite play \emph{belongs to} a strategy $\sigma$ for the first player if all its odd length finite initial plays are in $\sigma$.
A strategy for the first player is \emph{winning} if the first player wins in all infinite plays belonging to $\sigma$.
Similarly, one defines \emph{strategies} and \emph{winning strategies} for the second player.

\vspace{0.3 cm}

Finally, we summarise the matroid theoretic preliminaries. Given a matroid $M$, by $\Ccal(M)$ we denote the set of circuits of $M$, and by $\Scal(M)$ we denote the set of \emph{scrawls} of $M$, where a \emph{scrawl} is just any (possibly empty) union of circuits.\footnote{Matroids can be axiomatised in terms of their scrawls \cite{BC:rep_matroids}.} 
The {\em orthogonality axioms}, introduced in \cite{BC:determinacy}, are as follows, where 
$\Ccal$ and $\Dcal$ are sets of subsets of a groundset $E$, and can be thought of as the sets of circuits and cocircuits of some matroid, respectively.

\begin{itemize}
        \item[(O1)] $|C\cap D|\neq 1$ for all $C\in \Ccal$ and $D\in \Dcal$. 
        \item[(O2)] For all partitions $E=P\dot\cup Q\dot\cup \{e\}$
either $P+e$ includes an element of $\Ccal$ through $e$ or
$Q+e$ includes an element of $\Dcal$ through $e$.
\item[(O3)]For every $C\in \Ccal$, $e\in C$ and $X\subseteq E$, there is some $C_{min}\in \Ccal$ with $e\in C_{min} \se X \cup C$ such that $C_{min}\sm X$ is minimal.
\item [(O3$^*$)]For every $D\in \Dcal$, $e\in D$ and $X\subseteq E$, there is some $D_{min}\in \Dcal$ with $e\in D_{min} \se X \cup D$ such that $D_{min}\sm X$ is minimal.
\end{itemize}

\begin{thm}
[{\cite[Theorem 4.2]{BC:determinacy}}]\label{ortho_axioms+}
Let $E$ be a countable set and let $\Ccal,\Dcal\subseteq \Pcal(E)$.
Then there is a unique matroid $M$ such that $\Ccal(M)\subseteq \Ccal \subseteq \Scal(M)$
and $\Ccal(M^*)\subseteq \Dcal \subseteq \Scal(M^*)$ if and only if 
 $\Ccal$ and $\Dcal$ satisfy (O1), (O2), (O3) and (O3$^*$).
\end{thm}

We shall not be able to rely on this characterisation of matroids
since we will be dealing with possibly uncountable groundsets $E$.
So we will need an extra axiom.
A set $I\se E$ is \emph{independent} if it does not include any nonempty element of $\Ccal$.
Given $X\se E$, a \emph{base of $X$} is a maximal independent subset of $X$. A {\em base} of $(\Ccal, \Dcal)$ is a maximal independent subset of the ground set $E$.
\begin{itemize}
        \item[(IM)] 
Given an independent set $I$ and a superset $X$,
there exists a base of $X$ including $I$.
\end{itemize}

The proof of \autoref{ortho_axioms+} as in \cite{BC:determinacy} also proves the following:

\begin{cor}\label{ortho_axioms+_cor}
Let $E$ be a set and let $\Ccal,\Dcal\subseteq \Pcal(E)$.
Then there is a unique matroid $M$ such that $\Ccal(M)\subseteq \Ccal \subseteq \Scal(M)$
and $\Ccal(M^*)\subseteq \Dcal \subseteq \Scal(M^*)$ if and only if 
 $\Ccal$ and $\Dcal$ satisfy (O1), (O2), and (IM).
\end{cor}

\vspace{0.3 cm}

We say that $(\Ccal, \Dcal)$ is {\em tame} if the intersection of any set in $\Ccal$ with any set in $\Dcal$ is finite.
In the proof of our main result we will be in the situation that we have a 
pair $(\Ccal,\Dcal)$ of subsets of the powerset of some set $E$ that satisfies (O1) and (O2) and is tame. We call such a pair an \emph{orthogonality system}.

% there had been a typo in psi matroids paper: countable in the thm below not needed.

\begin{thm}
[{\cite[Theorem 4.4]{BC:determinacy}}]\label{O3_psi}
Any orthogonality system satisfies $(O3)$ and $(O3)^*$.
\end{thm}

\begin{rem}\label{base_char}
A set $B$ is a base for an orthogonality system $(\Ccal,\Dcal)$ if and only if
for each $x\notin B$, there is some $o\in \Ccal$ with $x\in o\se B+x$ and 
 for each $x\in B$, there is some $d\in \Dcal$ with $x\in d\se (E\sm B)+x$.
\end{rem}

Given $X\se E$, then the \emph{restriction $\Ccal\restric_X$ of $\Ccal$ to $X$} consists of those $o\in \Ccal$ included in $X$.
Similarly, the \emph{contraction $\Ccal.X$ of $\Ccal$ to $X$} is the set of those $a\se X$ such that
there is some $o\in \Ccal$ with $a=o\sm X$. 
We let $(\Ccal,\Dcal)\restric_X=(\Ccal\restric_X,\Dcal.X)$ and
$(\Ccal,\Dcal).X=(\Ccal.X,\Dcal\restric X)$. As usual we let $(\Ccal,\Dcal)\sm X= (\Ccal,\Dcal)\restric_{ (E\sm X)}$ and $(\Ccal,\Dcal)/ X= (\Ccal,\Dcal).(E\sm X)$.

\begin{rem}\label{minor_closed}
 If $(\Ccal,\Dcal)$ is an orthogonality system, then for any $X\se E$ both $(\Ccal,\Dcal)\restric_X$ and 
$(\Ccal,\Dcal).X$ are orthogonality systems.
\qed
\end{rem}

\begin{cor}\label{base}
Let $(\Ccal,\Dcal)$ be an orthogonality system such that for any two disjoint sets $A$ and $B$
the orthogonality system  $(\Ccal,\Dcal)/A\sm B$ has a base. Then $(\Ccal,\Dcal)$ satisfies (IM).
\end{cor}

\begin{proof}
Given $I$ and $X$ as in (IM), by assumption there is a base $B$ of $(\Ccal,\Dcal)/I\sm (E\sm X)$. It is straightforward to check that
$I\cup B$ is a base of $X$ with respect to $(\Ccal,\Dcal)$.
\end{proof}

\begin{comment}
 %Such a orthogonality system also satisfies (O3) and (O3$^*$).
\begin{proof}
 The first part of the lemma follows from \autoref{minor_closed}. TODO ???

\end{proof}

By \autoref{base}, we can strengthen \autoref{ortho_axioms+_cor} in that we can leave out 
(O3) and (O3$^*$) in there.
\end{comment}

In orthogonality systems we already have a notion of connectedness:
We say that two edges $e$ and $f$ are \emph{in the same connected component}
if there is some minimal nonempty $o\in \Ccal$ containing both $e$ and $f$.

\begin{lem}\label{connected_component}
 Being in the same component is an equivalence relation.

Moreover, $e$ is in the same connected component as $f$ if and only if there is some
minimal nonempty $d\in \Dcal$ containing both $e$ and $f$.
\end{lem}

\begin{proof}
 Just as for finite matroids.
\end{proof}

A \emph{connected component of $(\Ccal, \Dcal)$} is an equivalence class for the equivalence relation ``being in the same component''.

\begin{lem}\label{connected_component}
Given a connected component $X$ of $(\Ccal, \Dcal)$, then $(\Ccal, \Dcal)\restric_X=(\Ccal, \Dcal).X$.
\end{lem}

\begin{proof}
 Just as for finite matroids.
\end{proof}

\section{A simpler proof in a special case}\label{easy_o2}

The aim of this section is to give a simpler proof of 
a result from \cite{BC:determinacy}, \autoref{thm:loc_fin} below, which implies \autoref{ref2}. 
Given a locally finite graph $G$, a \emph{$\Psi$-circuit of $G$} is the edge set of a topological cycle of $G$ that only uses ends from $\Psi$.
A \emph{$\Psi\ct$-bond of $G$} is the edge set of a bond of $G$ that only has ends of $\Psi\ct$ in its closure.

\begin{thm}\label{thm:loc_fin}
Let $G$ be a locally finite graph and $\Psi$ a Borel set of ends.
Then there is a matroid $M_\Psi(G)$ whose circuits are the $\Psi$-circuits and whose cocircuits are the $\Psi\ct$-bonds. 
\end{thm}

Given a locally finite graph $G$ with a tree decomposition $(T,P(t)|t\in V(T))$, 
the \emph{torso $\bar P(t)$} of a part $P(t)$ is the multigraph obtained from $P(t)$
by adding for each neighbour $t'$ of $t$ and any two $v,w\in V(P(t))\cap V(P(t'))$
an edge $(tt',v,w)$ between $v$ and $w$. In this section, we shall only consider tree decompositions with each torso finite.
A \emph{precircuit} is a pair $(S,o)$ where $S$ is a connected subtree of $T$ and 
$o$ sends each $t\in V(S)$ to some $o(t)\se E(\bar P(t))$ that has even degree at each vertex such that for any neighbour $t'$ of $t$ in $T$ we have 
$o(t)\cap E(\bar P(t'))=\emptyset$ if $t'\notin V(S)$, and
$o(t)\cap E(\bar P(t'))=o(t')\cap E(\bar P(t))$ otherwise.
A \emph{precocircuit} is the same as a precircuit except that here $o(t)$ is a cut of $\bar P(t)$ instead of a set that has even degree at each vertex.
Given $\Psi\se \Omega(T)$, a precircuit $(S,o)$ is a \emph{$\Psi$-precircuit} if 
all ends of $S$ are in $\Psi$.
Similarly, one defines a $\Psi$-precocircuit.
We denote the set of underlying sets of $\Psi$-precircuits by $\Ccal_\Psi(G)$, and 
the set of underlying sets of $\Psi\ct$-precocircuits by $\Dcal_\Psi(G)$. 

The following follows from the fact that the finite circuits of any graph are the circuits of a matroid.

\begin{rem}\label{emptyset}
 The pair $(\Ccal_\emptyset(G),\Dcal_\emptyset(G))$ satisfies (O2).
\qed
\end{rem}

It suffices to prove \autoref{thm:loc_fin} for a graph $G'$ obtained from a locally finite graph $G$ by subdividing each edge. Indeed, then $M_\Psi(G)$ is a contraction minor of $M_\Psi(G')$. 
So from now on we fix a graph $G'$ obtained from a connected locally finite graph $G$ by subdividing each edge.\footnote{Mathematically, it is not strictly necessary to work with $G'$ instead of $G$, but 
this way we can cite a lemma from \cite{BC:determinacy}.}
We abbreviate $\Ccal_\Psi=\Ccal_\Psi(G')$ and $\Dcal_\Psi=\Dcal_\Psi(G')$.

\begin{lem}\label{td}[Lemma 7.5, Lemma 7.6 and Lemma 7.7 from \cite{BC:determinacy}]
There is a tree-decomposition $(T,P(t)|t\in V(T))$ of $G'$
with each part finite and a homeomorphism $\iota$ between $\Omega(G')$ and $\Omega(T)$
such that for each $\Psi\se \Omega(G')$, the set of minimal nonempty elements 
in $\Ccal_{\iota (\Psi)}$ is the set of $\Psi$-circuits of $G'$ and set of minimal nonempty elements 
in $\Dcal_{\iota(\Psi)}$ is the set of $\Psi\ct$-bonds of $G'$.

Furthermore, each $P(t)$ is connected and $T$ is locally finite.
\end{lem}

To simplify notation, we shall suppress the bijection $\iota$ from now on.
\begin{lem}\label{o1_tame}[Lemma 7.2 and Lemma 7.7 from \cite{BC:determinacy}]
The pair $(\Ccal_\Psi,\Dcal_\Psi)$ satisfies (O1) and tameness.
\end{lem}

\begin{lem}\label{to_cpsi}
The pair $(\Ccal_\Psi,\Dcal_\Psi)$ satisfies (O2) if and only if the  
pair consisting of the set of $\Psi$-circuits and the set of $\Psi\ct$-bonds does.
\end{lem}

\begin{proof}
One implication is obvious, for the other assume that
 $(\Ccal_\Psi,\Dcal_\Psi)$ satisfies (O1), (O2) and tameness.
Then by \autoref{O3_psi} $(\Ccal_\Psi,\Dcal_\Psi)$ satisfies (O3) and (O3$^*$).
Now let $E=P\dot\cup Q\dot\cup \{e\}$ be a partition.
If there is some $o\in \Ccal_\Psi$ with $e\in o\se P+e$, we can pick it minimal
by (O3), thus $o$ is a $\Psi$-circuit.
Otherwise by (O2), there is some $d\in \Dcal_\Psi$ with $e\in d\se Q+e$, and we conclude in the same way as above, which completes the proof.
\end{proof}

By \autoref{ortho_axioms+}, \autoref{o1_tame} and \autoref{to_cpsi}, the set of $\Psi$-circuits and the set of $\Psi\ct$-bonds are the sets of circuits and cocircuits of a matroid if and only if $\Ccal_\Psi$ and $\Dcal_\Psi$ satisfy (O2). Thus to prove \autoref{thm:loc_fin}, it suffices to show that $(\Ccal_\Psi,\Dcal_\Psi)$ satisfies (O2). 
So let $E(G')=P\dot\cup Q\dot\cup \{e\}$ be a partition.

We consider $T$ as a rooted tree rooted at the unique node $t_e$ such that $e\in E(P(t_e))$.
We consider the following positional game $(D,a,\Phi)$ played between two players called Sarah and Colin, where Sarah makes the first move.
$D$ is a directed graph whose underlying graph is bipartite with bipartition $(D_1,D_2)$.
The set $D_1$ is the union of sets $X(t)$, 
one for each $t\in V(T)$, where $X(t)$ is the set of those $F\se E(\bar P(t))\sm Q$ that have even degree at each vertex. 
The set $D_2$ contains the singleton of the starting vertex $a$ and includes the union of sets $Y(tt')$
one for each $tt'\in E(T)$, where $Y(tt')$ is the powerset of $E(\bar P(t))\cap E(\bar P(t'))$ without the empty set.  
We have a directed edge from $a$ to any $x\in X(t_e)$ containing $e$.
For $tt'\in E(T)$ directed away from the root, we have an edge from $x\in X(t)$ to $y\in Y(tt')$ if
$x\cap E(\bar P(t'))=y$. 
We have an edge from $y\in Y(tt')$ to $x\in X(t)$ if
$x\cap E(\bar P(t))=y$.

In order to complete the definition of the positional game, it remains to define $\Phi$. 
Given a play $Z$, by $Z[n]$ we denote the unique node $t$ with distance $n$ from $t_e$ such that the $2n+1\supst$ move of $Z$ is in $X(t)$.
For each infinite play $Z$, the sequence $(Z[n]|n\in \Nbb)$ is a ray of $T$ which belongs to some end $\omega_Z$. 
Let $f$ be the map from the space of infinite positional plays to the space of ends of $G$ defined via $f(Z)=\omega_Z$. 
It is straightforward to check that the map $f$ is continuous.
Let $\Phi$ be the inverse image of $\Psi$ under $f$. 
As being Borel is preserved under inverse images of continuous maps, we get the following.

\begin{rem}\label{pos_det}
 If $\Psi$ is Borel, then $\Phi$ is Borel and thus the positional game is determined.
\end{rem}

\begin{lem}\label{winning_strat_Sarah}
 If Sarah has a winning strategy, there is some $o\in \Ccal_\Psi$ with $e\in o\se P+e$.
\qed
\end{lem}

Given a winning strategy $\sigma$, then $S_{\sigma}$ is the induced subforest of $T$
whose nodes $t$ are those such that some $F\in X(t)$ appears as a move in a play according to $\sigma$ or are $t_e$. Note that $S_{\sigma}$ is connected.

\begin{lem}\label{winning_strat_Colin}
 If Colin has a winning strategy $\sigma$, then $S_{\sigma}$ 
has all ends in $\Psi\ct$.
\end{lem}

\begin{proof}
Assume that $S_\sigma$ has an end and let $\omega$ be an arbitrary end of $S_{\sigma}$.
Then there is a ray $t_1t_2...$ included in  $S_{\sigma}$ that belongs to $\omega$ with $t_1=t_e$.
For each $t_it_{i+1}$, we pick some play $P_i$ in $\sigma$
with Colin's last move in $Y(t_it_{i+1})$.

For each $t_it_{i+1}$, we denote by $Z_i$ the set 
of the initial plays of length $i$ of some $P_j$ with $j\geq i$.
As $T$ is locally finite and as each torso $\bar P(t)$ is finite, there are only finitely many 
possible plays ending with a move in $Y(t_it_{i+1})$. Hence $Z_i$ is finite.

We apply K\"onig's Infinity Lemma to the graph $H$ whose vertex set is the union of the $Z_i$. 
We join $a_i\in Z_i$ with $a_{i+1}\in Z_{i+1}$ by an edge if $a_{i+1}$ is an extension of $a_i$.
By the Infinity Lemma, we get a ray $x_1x_2...$ in $H$.
By construction $x_i$ is in $\sigma$. As $\sigma$ is winning, the union of all these plays is in $\Phi\ct$.
Thus $\omega$ is in $\Psi\ct$, which completes the proof.
\end{proof}

\begin{proof}[Proof of \autoref{thm:loc_fin}.]
As mentioned above, it suffices to show that $(\Ccal_\Psi,\Dcal_\Psi)$ satisfies (O2).
By \autoref{pos_det}, either Sarah or Colin has a winning strategy in the positional game above.
By \autoref{winning_strat_Sarah} we may assume that Colin has a winning strategy.

Let $H$ be the graph obtained from $G'$ by contracting all edges
not in any $P(t)$ with $t\in V(S_{ \sigma})$. 
We obtain $\tilde P(t)$ from $  P(t)$ by contracting all dummy edges
$(tt',v,w)$ with $t'\notin V(S_{ \sigma})$.
By the ``Furthermore''-part of \autoref{td}, it is clear that $H$ has a tree decomposition with tree $S_{ \sigma}$ whose torsos are the $\tilde P(t)$. %Let $\Psi_H=\Psi\cap \Omega(H)$. 

Suppose for a contradiction that there is some 
$o\in \Ccal_{\emptyset}(H)$ with $e\in o\se P+e$.
Let $(S,\bar o)$ be an $\emptyset$-precircuit with underlying set $o$.
If Sarah always plays $\bar o(t)$ in the positional game above, 
Colin always challenges her at some $v\in V(S_{ \sigma})$ when he plays according to $ \sigma$. As $S$ is rayless, eventually Colin cannot challenge, which contradicts the fact that  $ \sigma$ is winning.

Thus there cannot be such an $o$.
Thus by \autoref{emptyset}, there is some $d\in \Dcal_{\emptyset}(H)$ with $e\in d\se Q+e$.
Then $d\in \Dcal_{\Psi}(G')$ since all ends of $S_{ \sigma}$ are in $\Psi\ct$
by \autoref{winning_strat_Colin}. This completes the proof.
\end{proof}

\section{Simplifying winning strategies}\label{games}

In this section we prove \autoref{reduced} which allows us to improve winning strategies in positional games. Our proof of \autoref{main_intro} will rely on the determinacy of certain such games: this is why the set $\Psi$ is required to be Borel.

Given a set $\sigma$ of plays, by $\sigma(m)$ we denote  
the set of those moves that appear as $m$-th moves in plays of $\sigma$.
Given two finite or infinite plays  $P=p_1 \ldots $ and $Q=q_1 \ldots $ of the same length,
then $P\sim_1 Q$ if the first player makes the same moves in both plays, that is, 
$p_{i}=q_{i}$ for all odd $i$.
A winning strategy $\sigma$ for the first player is \emph{reduced} if 
there exists a total ordering $\leq$ of 
the set of positions with the following property: 
for any two plays $P=p_1...p_{2n+1}$ and $Q=q_1...q_{2n+1}$ in $\sigma$
such that $p_1...p_{2n-1}\sim_1 q_1...q_{2n-1}$ 
and $p_1...p_{2n}q_{2n+1}$ is a legal play, we have $p_{2n+1}\leq q_{2n+1}$.

\begin{lem}\label{reduced}
Let $\Gcal$ be a positional game whose set $\Phi$ of winning conditions is closed under $\sim_1$.
If the first player has a winning strategy $\sigma$, then 
the first player has a reduced winning strategy.
\end{lem}

\begin{proof}
First, we pick a well-order $\leq$ of the set of positions.
Next we define a reduced winning strategy $\bar \sigma$ for the first player.
 The first player should play as follows.
 His first move should be the same as in $\sigma$. 
Whenever he has just made a move he should have in mind an auxiliary play according to $\sigma$
which ends at the same position.
Assume that the first player and the second player have already played $2n+1$ moves and 
let $s$ be the current play, and $s'$ the auxiliary play.
Now assume that the second player's response is $m$.
As $\sigma$ is winning, there is a move $t$ of the first player such that $s'mt\in \sigma$.
Let $X$ be the set of those pairs $(m',u)$
such that $s'm'u\in \sigma$ and 
$smu$ is a legal play. $X$ is nonempty since $(m,t)\in X$.
The first player picks $(m',u)\in X$ such that $u$ is minimal with respect to $\leq$ and, subject to this, such that $m'$ is minimal with respect to $\leq$.
He plays $u$ and imagines the auxiliary play $(smu)'=s'm'u$.

It is clear that this defines a strategy for the first player.
Next, we show $\bar \sigma$ is winning.
So let $(s_n|n\in \Nbb)$ be a sequence of plays according to $\bar \sigma$ with $s_n$ of length $2n+1$, each extending the previous one. 
By construction, it is clear that $s_{n+1}'$ extends $s_n'$.
As $s_n'\in \sigma$, the union of the $s_n'$ is in 
$\Phi$. As $\Phi$ is closed under $\sim_1$, the union of the $s_n$ is in 
$\Phi$ as well. Thus $\bar \sigma$ is winning.

By induction, it is straightforward to check that
if $s,t\in \bar \sigma$ and $s\sim_1 t$, then $s'=t'$.

It remains to show that $\bar \sigma$ is reduced.
So let $P=p_1...p_{2n+1}$ and $Q=q_1...q_{2n+1}$ in $\bar\sigma$
with $p_1...p_{2n-1}\sim_1 q_1...q_{2n-1}$ such that $p_1...p_{2n}q_{2n+1}$ is a legal play.
Let $s = p_1 \ldots p_{2n-1}$. 
Let $Q'=u_1...u_{2n+1}$.
Then as noted above we have $s' = u_1 \ldots u_{2n-1}$, so $s'u_{2n}u_{2n+1} \in \sigma$.
So by the construction of $p_{2n+1}$ we have $p_{2n+1} \leq u_{2n+1} = q_{2n+1}$.
Thus $\bar \sigma$ is reduced, which completes the proof.
\end{proof}

The following is a direct consequence of the definition of a reduced strategy.

\begin{rem}\label{change_a-little}
Let $\bar\sigma$ be a reduced winning strategy for the first player.
Let $p_1...p_n\in \bar \sigma$ and $q_1...q_m\in \bar\sigma$ and assume there is some odd $i$ such that $p_1...p_i\sim_1 q_1...q_i$. Then $p_1...p_{i}q_{i+1}...q_m\in \bar \sigma$.
\qed
\end{rem}

\begin{comment}

Given two finite or infinite plays  $P=p_1 \ldots $ and $Q=q_1 \ldots $ of the same length,
then $P\sim_2 Q$ if the second player makes the same moves in both plays, that is, 
$p_{i}=q_{i}$ for all even $i$.
Reduced strategies for the second player are defined like those for the first player with $\sim_1$ replaced by $\sim_2$ and $2n+1$ replaced by $2n+2$.
An argument like that for \autoref{reduced} gives the following.

\begin{lem}\label{reduced2}
Let $\Gcal$ be a positional game with winning condition $\Phi$
that is closed under $\sim_2$.
If the second player has a winning strategy $\sigma$, then 
the second player has a reduced winning strategy.
\end{lem}

\end{comment}

\section{Presentations} \label{pres}
Fix a field $k$. For any set $E$ and any element $v$ of the vector space $k^E$, the {\em support} $\supp(v)$ is $\{e \in E | v(e) \neq 0\}$. To simplify our notation, we formally consider such a vector $v$ to be a function with domain the support of $v$. This means that we can also consider $v$ itself to be a member of other vector spaces $k^F$ with $\supp(v) \subseteq F$. Note that addition and scalar multiplication of vectors are unambiguous with respect to this convention. If $V$ is a subspace of $k^E$, we denote by $S(V)$ the set of supports of vectors in $V$.

For $v, w \in k^E$ we say that $v$ and $w$ are {\em orthogonal}, denoted $v \perp w$, if $\sum_{e \in E}v(e)w(e) = 0$. Here and throughout the paper such equalities are taken to include the claim that the sum on the left is well defined, in the sense that only finitely many summands are nonzero. That is, if $v \perp w$ then in particular $\supp(v) \cap \supp(w)$ is finite. If $V$ and $W$ are subspaces of $k^E$ then we say they are {\em orthogonal}, denoted $V \perp W$, if $(\forall v \in V)(\forall w \in W) v \perp w$.

As in \cite{BC:determinacy}, we will need some extra linear structure over $k$ to allow us to stick together matroids along sets of dummy edges of size more than 1. In fact, we will stick together presentations over $k$ of the matroids in question to obtain a presentation of the resulting matroid. Therefore we must specify precisely what objects we will be taking as presentations of infinite matroids over $k$.

\begin{dfn}
Let $E$ be any set.
A {\em presentation} $\Pi$ on $E$ consists of a pair $(V, W)$ of orthogonal subspaces of $k^E$ such that $S(V)$ and $S(W)$ satisfy (O2). Elements of $V$ are called {\em vectors} of $\Pi$ and elements of $W$ are called {\em covectors}. We will sometimes denote the first element of $\Pi$ by $V_{\Pi}$ and the second by $W_{\Pi}$. We say that $\Pi$ {\em presents} the matroid $M$ if the circuits of $M$ are the minimal nonempty elements of $S(V_{\Pi})$ and the cocircuits of $M$ are the minimal nonempty elements of $S(W_{\Pi})$.
\end{dfn}

It is clear from the results of \cite{THINSUMS} that a tame matroid $M$ is representable in the sense of that paper over $k$ if and only if there is a presentation over $k$ which presents $M$.

Note that for any presentation $\Pi$, the pair $(S(V_{\Pi}), S(W_{\Pi}))$ is an orthogonality system. Accordingly, we say a set is {\em $\Pi$-independent} when it is independent with respect to this orthogonality system.

\begin{rem}\label{countE}
If $E$ is a countable set then any presentation on $E$ presents a matroid by \autoref{ortho_axioms+}.
\end{rem}

\begin{dfn}
If $V$ is a subspace of $k^E$ then for $X$ a subset of $E$ we define the {\em restriction} $V \restric_X$ of $V$ to $X$ to be $\{v \in V | \supp(v) \subseteq X\}$. We denote the restriction of $V$ to $E \setminus Q$ by $V \backslash Q$, and say it is obtained from $V$ by {\em removing} $Q$. Similarly, for $X$ a subset of $E$ we define the {\em contraction} $V.X$ of $V$ to $X$ to be $\{v \restric_X | v \in V\}$. We denote the contraction of $V$ to $E \setminus P$ by $V / P$, and say it is obtained from $V$ by {\em contracting} $P$. We also define these terms for presentations as follows:
\begin{eqnarray*}
(V, W) \restric_X &=& (V \restric_X, W.X) \\
(V, W) \backslash Q &=& (V \backslash Q, W / Q) \\
(V, W).X &=& (V.X, W \restric_X) \\
(V, W) / P &=& (V / P, W \backslash P)\\
\end{eqnarray*} 
All of these operations give rise to new presentations, called {\em minors} of the original presentation.
\end{dfn}

We will need some basic lemmas about presentations.

\begin{lem}\label{finpres}
Let $E$ be a finite set. Then a pair $(V, W)$ of subspaces of $E$ is a presentation on $E$ if and only if $V$ and $W$ are orthogonal complements.
\end{lem}
\begin{proof}
For any subspace $U$ of $k^E$, we will denote the orthogonal complement of $U$ by $U^{\perp}$.
Suppose first of all that $W = V^{\perp}$. We must show that $S(V)$ and $S(W)$ satisfy (O2), so suppose we have a partition $E = P \dot \cup Q \dot \cup \{e\}$. If there is no $v \in V$ with $e \in supp(v) \subseteq P + e$ then $\mathbb{1}_e \not \in V + k^P$, so $V^{\perp} \cap (k^P)^{\perp} \not \subseteq \langle \mathbb{1}_e \rangle^{\perp}$. That is, there is some $w$ which is in $V^{\perp} = W$ and is in $(k^P)^{\perp}$, so that $\supp(w) \subseteq Q + e$, but with $w(e) \neq 0$. Thus $e \in \supp(w) \subseteq Q + e$, as required.

Now suppose that $(V, W)$ is a presentation, so that $S(V)$ and $S(W)$ satisfy (O2). Suppose for a contradiction that $W \neq V^{\perp}$, so that there is some $w \in V^{\perp} \setminus W$. As $E$ is finite, we can choose such a $w$ with minimal support. Since $w \neq 0$, we can pick some $e \in \supp(w)$. Let $P = E \setminus \supp(w)$ and $Q = \supp(w) - e$. Applying (O2) to the partition $E = P \dot \cup Q \dot \cup \{e\}$, we either get some $v \in V$ with $e \in \supp(v) \subseteq P + e$, so that $\supp(v) \cap \supp(w) = \{e\}$, contradicting our assumption that $w \in V^{\perp}$, or else we get some $w' \in W$ with $e \in \supp(w') \subseteq Q + e$. But in that case, letting $w'' = w - \frac{w(e)}{w'(e)}w'$ we have that $w'' \in V^{\perp}$ and $\supp(w'') \subsetneq \supp(w)$, so by minimality of the support of $w$ we have $w'' \in W$. Thus $w = w'' +  \frac{w(e)}{w'(e)}w' \in W$, which is again a contradiction.
\end{proof}

\begin{dfn}
Let $\Pi = (V, W)$ be a presentation on a set $E$ and $x \in k^E$. Then $\Pi_x = (V_x, W^x)$ is the pair of orthogonal subspaces of $k^{E + *}$ given by $V_x = V + \langle x - \mathbb{1}_{*}\rangle$ and $W_x = \{w \in k^{E + *}| w \restric_E \in W \text{ and } w(*) = \sum_{e \in E} w(e)x(e)\}$.
\end{dfn}

\begin{rem}\label{xmincom}
If $P$ and $Q$ are disjoint subsets of $E$ not meeting $\supp(x)$ then $$\Pi_x /P \backslash Q = (\Pi / P \backslash Q)_x\,.$$
\end{rem}

If $E$ is finite then it is clear (using the equivalent characterisation of presentations in \autoref{finpres}) that $\Pi_x$ is again a presentation. In fact this is more generally true:

\begin{lem}
Let $\Pi = (V, W)$ be a presentation on a set $E$ and let $x \in k^E$ have finite support. Then $\Pi_x$ is a presentation on $E + *$.
\end{lem}
\begin{proof}
It is clear that $V_x \perp W^x$, so we just have to prove (O2). Suppose we have some partition $E + * = P \dot \cup Q \dot \cup \{e\}$. Let $F$ be the finite set $\supp(x) + e + *$. Now consider the presentation $\Pi' = (\Pi / (P \setminus F)  \backslash (Q  \setminus F))_x$ on the finite set $F$. By \autoref{xmincom}, we have $\Pi' = \Pi_x / (P \setminus F) \backslash (Q \setminus F)$. We now apply (O2) in $\Pi'$ to the partition $F = (P \cap F) \dot \cup (Q \cap F) \dot \cup \{e\}$. If we find a vector $v$ of $\Pi'$ with $e \in \supp(v) \subseteq (P \cap F) + e$ then we can extend $v$ to a vector $v'$ of $\Pi_x$ which witnesses (O2) in that $e \in \supp(v) \subseteq P + e$. The case that there is a covector $w$ of $\Pi$ with $e \in \supp(w) \subseteq (Q \cap F) + e$ is dealt with similarly.
\end{proof}

By \autoref{O3_psi}, for any presentation $(V, W)$ we must have that $S(V)$ satisfies (O3). We are now in a position to prove a more general (O3)-like principle.

\begin{lem}\label{O3pres}
Let $\Pi$ be a presentation on a set $E$, $v_0$ a vector of $\Pi$, $X$ a subset of $E$ and $F$ a finite subset of $E$ disjoint from $X$. Then amongst the set $L_{v_0}^{F, X}(\Pi)$ of vectors $v$ of $\Pi$ such that $v \restric_F = v_0 \restric_F$ and $\supp(v) \subseteq \supp(v_0) \cup X$, there is one with $\supp(v) \setminus X$ minimal.
\end{lem}
\begin{proof}
We put a preordering on $L_{v_0}^{F,X}(\Pi)$ by $v \leq v'$ if $\supp(v) \setminus X \subseteq \supp(v') \setminus X$.
The function $v \mapsto v - v_0\restric_F + \mathbb1_*$ is an order-preserving bijection from $L_{v_0}^{F,X}(\Pi)$ to $L_{v_0 - v_0\restric_F + \mathbb1_*}^{\{*\}, X}(\Pi_{v_0 \restric_F})$. The latter collection has a minimal element by (O3) applied to the set of supports of vectors of the presentation $\Pi_{v_o \restric_F}$. Hence the former collection also has a minimal element.
\end{proof}

\begin{rem}\label{minisind}
Let $v\in L_{v_0}^{F, X}(\Pi)$ such that $\supp(v) \setminus X$ is minimal.
Then the set $\supp(v) \setminus X$ is $\Pi/X$-independent.
\end{rem}

\begin{cor}\label{fromO3}
Let $\Pi$ be a presentation on a set $E$, $F$ a finite subset of $E$ and $P$ a subset of $E$ disjoint from $F$. Then there is a $\Pi$-independent subset $P'$ of $P$ such that $(\Pi / P)\restric_F = (\Pi / P') \restric_F$.
\end{cor}
\begin{proof}
We successively apply \autoref{O3pres} and \autoref{minisind} to elements of a base of $(\Pi / P) \restric_F$.
\end{proof}

\section{Trees of presentations} \label{treepres}

We can now mimic the construction of \cite{BC:determinacy} to glue together trees of presentations.

\begin{dfn}
A {\em tree of presentations} $\Tcal$ consists of a tree $T$, together with functions $\overline V$ and $\overline W$ assigning to each node $t$ of $T$ a presentation $\Pi(t) = (\overline V(t), \overline W(t))$ on the ground set $E(t)$, such that for any two nodes $t$ and $t'$ of $T$, $E(t) \cap E(t')$ is finite and if $E(t) \cap E(t')$ is nonempty then $tt'$ is an edge of $T$.

For any edge $tt'$ of $T$ we set $E(tt') = E(t) \cap E(t')$. We also define the {\em ground set} of $\Tcal$ to be $E = E(\Tcal) = \left(\bigcup_{t \in V(T)} E(t)\right) \setminus \left(\bigcup_{tt' \in E(T)} E(tt')\right)$. 

We shall refer to the edges which appear in some $E(t)$ but not in $E$ as {\em dummy edges} of $M(t)$: thus the set of such dummy edges is $\bigcup_{tt' \in E(T)} E(tt')$.
\end{dfn}

In sticking together such a tree of presentations, we shall make use of some additional information, namely a set $\Psi$ of ends of $T$. We think of the ends in $\Psi$ as being available to be used by the new vectors and those in the complement $\Psi\ct$ of $\Psi$ as being available to be used by the new covectors. More formally:

\begin{dfn}
Let $\Tcal = (T, \overline V, \overline W)$ be a tree of presentations. A {\em pre-vector} of $\Tcal$ is a pair $(S, \overline v)$, where $S$ is a subtree of $T$ and $\overline v$ is a function sending each node $t$ of $S$ to some $\overline v(t) \in \overline V(t)$, such that for each $t\in S$ we have 
$\overline v(t) \restric_{E(tu)} = \overline v(u) \restric_{E(tu)}\neq 0$ if $u\in S$, and $\overline v(t) \restric_{E(tu)}=0$ otherwise. 
 The {\em underlying vector} $\underline{(S, \overline v)}$ of $(S, \overline v)$ is the element of $k^{E(\Tcal)}$ which at a given $e \in E(\Tcal)$ takes the value $\overline v(t)(e)$ if there is some $t \in S$ with $e \in E(t)$, and otherwise takes the value 0. The {\em support} $\supp(S, \overline v)$ of a pre-vector is the support of the underlying vector.

Now let $\Psi$ be a set of ends of $T$. A pre-vector $(S, \overline v)$ is a {\em $\Psi$-pre-vector} if all ends of $S$ are in $\Psi$. The space $V_{\Psi}(\Tcal)$ of {\em $\Psi$-vectors} is the subspace of $k^E$ generated\footnote{under finite linear combinations} by the underlying vectors of $\Psi$-pre-vectors. 

A {\em pre-covector} of $\Tcal$ is a pair $(S, \overline w)$, where $S$ is a subtree of $T$ and $\overline w$ is a function sending each node $t$ of $S$ to some $\overline w(t) \in \overline W(t)$, such that for each $t\in S$ we have 
$\overline w(t) \restric_{E(tu)} = - \overline w(u) \restric_{E(tu)} \neq 0$ if $u\in S$, and $\overline w(t) \restric_{E(tu)}=0$ otherwise (note the change of sign from the definition of pre-vectors). Underlying covectors and supports are defined as above. A pre-covector $(S, \overline w)$ is a {\em $\Psi$-pre-covector} if all ends of $S$ are in $\Psi$. The space $W_{\Psi}(\Tcal)$ of {\em $\Psi\ct$-covectors} is the subspace of $k^E$ generated by the underlying covectors of $\Psi\ct$-pre-covectors. Finally, $\Pi_{\Psi}(\Tcal)$ is the pair $(V_{\Psi}(\Tcal), W_{\Psi}(\Tcal))$. We may omit the subscripts from $V_{\Psi}(\Tcal)$, $W_{\Psi}(\Tcal)$ and $\Pi_{\Psi}(\Tcal)$ if the set of ends of $T$ is empty.
\end{dfn}

\begin{rem}\label{minors}
Let $P$ and $Q$ be sets which don't meet any of the sets $E(tu)$ with $tu$ an edge of $T$. Then $\Pi_{\Psi}(\Tcal)/P\backslash Q = \Pi_{\Psi}(T, \overline V / P \backslash Q, \overline W \backslash P / Q)$, where $\overline V / P \backslash Q \colon t \mapsto \overline V (t) /P \backslash Q$ and $\overline W \backslash P / Q \colon t \mapsto \overline W(t) \backslash P /Q$.
\end{rem}

Our notation suggests that $V_{\Psi}(\Tcal)$ and $W_{\Psi}(\Tcal)$ should be orthogonal. This is often true, but as the following example shows some extra restriction is needed to ensure that intersections of supports of vectors with supports of covectors are finite.

\begin{eg}\label{cex}
Let $(V, W)$ be any presentation having some vector $v$ of infinite support and some covector $w$ of infinite support. Let $(e_i | i \in \Nbb)$ be an infinite sequence of distinct elements of $\supp(v) \setminus \supp(w)$ and $(f_i | i \in \Nbb)$ an infinite sequence of distinct elements of $\supp(w) \setminus \supp(v)$. We also introduce for each $i \in \Nbb$ the presentation $\Pi_i = (V_i, W_i)$ on ground set $E_i = (e_i, f_i, g_i, h_i)$, where the $g_i$ and $h_i$ are all chosen distinct and outside $E$, and where $V_i = \{v \in k^{E_i} | v(f_i) = 0 \text{ and } v(g_i) = v(h_i)\}$ and $W_i = \{w \in k^{E_i} | w(e_i) = 0 \text{ and } w(g_i) = -w(h_i)\}$. Let $v_i \in V_i$ be the vector taking the value $v(e_i)$ at $e_i$, $g_i$ and $h_i$ and 0 at $f_i$. Let $w_i \in W_i$ be the covector taking the value 0 at $e_i$, $-w(f_i)$ at $f_i$ and $g_i$ and $w(f_i)$ at $h_i$.

Let $T$ be the star with central node $*$ and whose leaves are the natural numbers. Then we get a tree of presentations $\Tcal = (T, \overline V, \overline W)$ by letting $\overline V(*) = V$ and $\overline V(i) = V_i$ for each $i \in \Nbb$ and defining $\overline W$ similarly. We get a pre-vector $(T, \overline v)$ by letting $\overline v(*) = v$ and $\overline v(i) = v_i$ and a pre-covector $(T, \overline w)$ by letting $\overline w(*) = w$ and $\overline w(i) = w_i$. Then the intersection of the supports of $(T, \overline v)$ and $(T, \overline w)$ includes $\bigcup_{i \in \Nbb}\{g_i, h_i\}$, and so is infinite.
\end{eg}

In order to avoid this sort of situation, we introduce the following restriction:

\begin{dfn}
Let $\Pi$ be a presentation on a set $E$, and let $\Fcal$ be a set of disjoint subsets of $E$. We say that $\Pi$ is {\em neat} with respect to $\Fcal$ if for any $v \in V_{\Pi}$ and $w \in W_{\Pi}$ there are only finitely many $F \in \Fcal$ meeting the supports of both $v$ and $w$. We say that a tree $\Tcal = (T, \overline V, \overline W)$ of presentations is {\em neat} if for each node $t$ of $T$ the presentation $\Pi(t)$ is neat with respect to the set of sets $E(tu)$ with $u$ adjacent to $t$ in $T$.
\end{dfn}

\begin{lem}
Let $\Tcal = (T, \overline V, \overline W)$ be a neat tree of presentations, and $\Psi$ a set of ends of $T$. Then $V_{\Psi}(\Tcal) \perp W_{\Psi}(\Tcal)$. 
\end{lem}
\begin{proof}
It suffices to show that for any $\Psi$-pre-vector $(S, \overline v)$ and any $\Psi\ct$-pre-covector $(S', \overline w)$ we have $\underline{(S, \overline v)} \perp \underline{(S', \overline w)}$. All ends of the tree $S \cap S'$ must be in $\Psi \cap \Psi \ct = \emptyset$: that is, $S \cap S'$ is rayless. Since $\Tcal$ is neat, each vertex of $S \cap S'$ has finite degree in $S \cap S'$. 
Thus by K\"onig's Lemma the tree $S \cap S'$ is finite. The intersection of the supports of $(S, \overline v)$ and $(S', \overline w)$ is a subset of the finite set $\bigcup_{t \in S \cap S'} (\supp(\overline v(t)) \cap \supp(\overline w(t)))$ and so is finite.

For any edge $tu$ of $S \cap S'$ and $e \in E(tu)$ we have $\overline v(t)(e) \overline w(t)(e) + \overline v(u)(e) \overline w(u)(e) = 0$, and so we have
$$\sum_{e \in E} \underline{(S, \overline v)}(e) \underline{(S', \overline w)}(e) = \sum_{t \in S \cap S'}\sum_{e \in E(t)} \overline v(t)(e) \overline w(t)(e) = 0 \,.$$
\end{proof}

However, our aim is to use the construction of $\Pi_{\Psi}(\Tcal)$ to produce matroids, so we are also interested in the question of when 
$\Pi_\Psi(\Tcal)$ presents a matroid, that is, the minimal nonempty $\Psi$-vectors and the  minimal nonempty $\Psi\ct$-covectors satisfy (O2) and (IM). It is not even clear that our construction will yield matroids when applied to the simplest sorts of trees, namely stars with all leaves finite. More precisely:

\begin{dfn}
Let $\Pi$ be a presentation on a set $E$ and let $\Fcal$ be a set of disjoint subsets of $E$. An {
\em $\Fcal$-star} of presentations around $\Pi$ is a tree $(T, \overline V, \overline W)$ of presentations where $T$ is the star with central node $*$ and leaf set $\Fcal$, $(\overline V(*), \overline W(*)) = \Pi$, and for each $F \in \Fcal$ the set $E(F)$ is finite and $E(*F) = F$. We say that $\Pi$ is {\em stellar} with respect to $\Fcal$ if for any $\Fcal$-star $\Tcal$ of presentations around $\Pi$, the pair $\Pi_{\emptyset}(\Tcal)$ is a presentation and presents a matroid. We say that a tree $\Tcal = (T, \overline V, \overline W)$ of presentations is {\em stellar} if for each node $t$ of $T$ the presentation $\Pi(t)$ is stellar with respect to the set of sets $E(tu)$ with $u$ adjacent to $t$ in $T$.
\end{dfn}

\begin{rem}\label{stellareg}There are many examples of stellar presentations. For example, if $\Pi$ is finitary\footnote{A presentation $(V,W)$ is \emph{finitary} if every element in $V$ is finite. } or $\Fcal$ is finite then $\Pi$ is stellar with respect to $\Fcal$. If $\Pi'$ is a minor of $\Pi$ on the set $E'$ and $\Pi$ is stellar with respect to $\Fcal$ then $\Pi'$ is stellar with respect to $\{F \cap E' | F \in \Fcal\}$. Furthermore, if $\Pi$ is stellar with respect to $\Fcal$ and $\Fcal'$ is a set of disjoint sets such that each $F' \in \Fcal'$ is a subset of some $F \in \Fcal$ then $\Pi$ is also stellar with respect to $\Fcal'$. This fact, together with the construction given in \autoref{cex}, shows that if $\Pi$ is stellar with respect to $\Fcal$ then it is necessarily also neat with respect to $\Fcal$.
\end{rem}

Our strategy, aiming at maximal generality, is to leave the question of precisely which presentations are stellar open but to reduce the question of when sticking together trees of presentations gives a presentation of a matroid to this problem. That is, we shall show that if $\Tcal$ is a stellar tree of presentations and $\Psi$ is a Borel set of ends then $\Pi_{\Psi}(\Tcal)$ is a presentation of a matroid. (O2) will be proved in \autoref{O2} and (IM) in \autoref{IM}. We note, however, that the following question remains open:

\begin{oque}
If a presentation $\Pi$ is neat with respect to some $\Fcal$, must it also be stellar with respect to $\Fcal$?
\end{oque}

We will rely on the following straightforward rearrangement of the definition of stellarity:

\begin{lem}\label{stellagain}
Let $\Pi = (V, W)$ be a presentation on a set $E$ which is stellar with respect to $\Fcal \subseteq \Pcal(E)$, and let $F_0 \in \Fcal$ and $w_0 \in k^{F_0}$. Let $Q$ be a set disjoint from all $F \in \Fcal$. For each $F \in \Fcal - F_0$ let $W(F)$ be a subset of $k^F$. Suppose that for every $v \in V$ with $v \not \perp w_0$ and $\supp(v) \cap Q = \emptyset$ there is some $F \in \Fcal - F_0$ and some $w \in W(F)$ such that $w \not \perp v$. Then there is some $w \in W$ such that $w \restric_{F_0} = w_0$, $\supp(w) \subseteq Q \cup \bigcup \Fcal$, and for each $F \in \Fcal - F_0$ we have $w \restric_F \in \langle W(F) \rangle$.
\end{lem}
\begin{proof}
Without loss of generality each $W(F)$ is a subspace of the corresponding space $k^F$. Let $V(F) = W(F)^{\perp}$, and $\Pi(F) = (V(F), W(F))$, which is a presentation by \autoref{finpres}. Let $\Tcal = (T, \overline V, \overline W)$ be the $(\Fcal - F_0)$-star of presentations around $\Pi$, where the presentation at the leaf $F$ is $\Pi(F)$. Since $\Pi$ is stellar, $\Pi_{\emptyset}(\Tcal)$ is a presentation. Now we consider the presentation $(\Pi_{\emptyset}(\Tcal) \backslash Q).F_0$, which by \autoref{finpres} consists of a pair $(V_0, W_0)$ of complementary subspaces of $k^{F_0}$. What we have to prove is just that $w_0 \in W_0$.

Suppose not for a contradiction. Then there is some $v_0 \in V_0$ with $v_0 \not \perp w_0$. By definition this $v_0$ must arise as $\overline v(*) \restric_{F_0}$ for some pre-vector $(S, \overline v)$ of $\Tcal$ whose support does not meet $Q$. Then $\overline v(*) \not \perp w_0$, so there is some $F \in \Fcal - F_0$ and some $w \in W(F)$ such that $w \not \perp \overline v(*) \restric_F = \overline v(F)$, contradicting the fact that $\overline v(F) \in V(F)$.
\end{proof}

\section{(O2) for trees of presentations}\label{O2}

Our aim in this section is to show that, for any stellar tree $\Tcal = (T, \overline V, \overline W)$ of presentations and any Borel set $\Psi$ of ends of $T$, the sets $S(V_{\Psi}(\Tcal))$ and $S(W_{\Psi}(\Tcal))$ satisfy (O2). Thus we begin by fixing such a $\Tcal$ and $\Psi$. We also fix a partition $E(\Tcal) = P \dot \cup Q \dot \cup \{e\}$. We shall consider the vertex $t_0$ of $T$ with $e \in E(t_0)$ to be the root of $T$, and we consider $T$ as a directed graph with the edges directed away from $t_0$. To prove (O2), it suffices to prove that there is either a $\Psi$-pre-vector $(S, \overline v)$ with $e \in \supp(S, \overline v) \subseteq P+e$ or else a $\Psi\ct$-pre-covector $(S, \overline w)$ with $\in \supp(S, \overline w) \subseteq Q+e$. To this end, we recall two games, called the {\em circuit game} and {\em cocircuit game}, from \cite{BC:determinacy}. To match the formalism of \autoref{games}, we shall present these games as positional games.

To simplify notation in this section, we shall not distinguish
between an end $\omega$ of a rooted tree $T$ and the unique ray belonging to $\omega$ that starts at the root.

\begin{dfn}
Let $X$ be the set of pairs $(t, v)$ with $t$ a vertex of $T$ and $v \in \overline V(t)$ such that $\supp(v) \cap Q = \emptyset$. Let $Y$ be the set of pairs $(tu, w)$ with $tu$ an edge of $T$ and $w \in k^{E(tu)}$. 

The {\em circuit game} $\Gcal = \Gcal(T, \overline V, \overline W, \Psi, P, Q)$ is the positional game played on the digraph $D$ with vertex set  $X \sqcup Y \sqcup \{a\}$ and with edges given as follows:
\begin{itemize}
\item an edge from $a$ to $(t_0, v) \in X$ when $e \in \supp(v)$.
\item an edge from $(t, v) \in X$ to $(tu, w) \in Y$ when $v \not \perp w$.
\item an edge from $(tu, w) \in Y$ to $(u, w) \in X$ when $v \not \perp w$.
\end{itemize}

Any infinite walk from an outneighbour  of $a$ in $D$ induces an infinite walk from $t_0$ in $T$, which is an end of $T$. The set $\Phi$ of winning conditions of $\Gcal$ is the set of infinite walks from outneighbours of $a$ in $D$ which induce walks to ends in $\Psi$. We call the two players of the circuit game Sarah and Colin, with Sarah playing first.

The {\em cocircuit game} is the game like the dual circuit game $\Gcal(T, \overline W, \overline V, \Psi\ct, Q, P)$ but with the roles of Sarah and Colin reversed.
\end{dfn}

It is not hard to see that this definition is just a reformulation of \cite[Definition 8.1]{BC:determinacy}. Using the arguments of that paper, we may now obtain the following results:

\begin{lem}
Either Sarah or Colin has a winning strategy in the circuit game.\qed
\end{lem}

\begin{lem}
Colin has a winning strategy in the circuit game if and only if he has one in the cocircuit game.
\end{lem}
\begin{proof}
Just like the proof of \cite[Lemma 8.5]{BC:determinacy}, but using \autoref{stellagain} in place of \cite[Sublemma 8.6]{BC:determinacy}
\end{proof}

From now on we shall assume that Sarah has a winning strategy $\sigma$ in the circuit game: the argument if Colin has a winning strategy there is dual to the one which follows. Let $S_{\sigma}$ be the subtree of $T$ consisting of those vertices $t$ for which there is some $v$ such that Sarah might at some point play $(t, v)$ when playing according to $\sigma$. We would like to mimic the argument of \cite[Lemma 8.2]{BC:determinacy} to construct a $\Psi$-precircuit from $\sigma$. In order to do this, we would need all ends of $S_{\sigma}$ to be in $\Psi$. Although there is no reason to expect this to happen in general, it will happen if $\sigma$ is reduced.

\begin{lem}
Let $\sigma$ be a reduced winning strategy in the circuit game, and let $S_{\sigma}$ be defined as above. Then all ends of $S_{\sigma}$ are in $\Psi$.
\end{lem}
\begin{proof}
For any finite sequence $s$ we denote the last element of $s$ by $l(s)$.
For any finite play $s$ in $\Gcal$, let $\hat s$ be the sequence of moves played by Sarah in $s$ (that is, the sequence $(s_{2k + 1} | 0 \leq k \leq \text{length}(s)/2)$). Let $\tau = \{\hat s | s \in \sigma\}$. 

First of all we will show that for any edge $tu$ of $T$ and any $s \in \tau$ with $\pi_1(l(s)) = t$ there are no more than $|E(tu)|$ extensions $s' \in \tau$ of $s$ with $\pi_1(l(s')) = u$. Suppose for a contradiction that there are more than this. Then each such $s'$ gives rise to a vector $\pi_2(l(s')) \restric_{E(tu)}$ in $k^{E(tu)}$, and there must be some linear dependence of these vectors. So suppose that $\sum_{i = 1}^n \lambda_i \pi_2(l(s^i))\restric_{E(tu)} = 0$, where for each $i$ $\lambda_i$ is nonzero and $s^i$ is an extension of $s$ in $\tau$ with $\pi_1(l(s^i)) = u$. 
Let $k$ be the length of $s$, and let $j$ be such that $l(s^j)$ is maximal in the order $\leq$. Without loss of generality $j = n$. Let $s' = s'_1 ... s'_{2k +1} \in \sigma$ with $\widehat {s'} = s^n$. Then 
$$\pi_2(s'_{2k}) \not \perp \pi_2(l(s^n)) \restric_{E(tu)} = -\frac1{\lambda_n}\sum_{i = 1}^{n-1} \lambda_i \pi_2(l(s^i))\restric_{E(tu)} \, ,$$
so there is some $i < n$ with $\pi_2(s'_{2k}) \not \perp \pi_2(l(s^i)) \restric_{E(tu)}$. But then $s'_1s'_2...s'_{2k}l(s^i)$ is a legal play in $\Gcal$ and $l(s^i) < s'_{2k+1}$, contradicting our assumption that $\sigma$ is reduced. Thus there are at most $|E(tu)|$ (and in particular only finitely many) extensions $s' \in \tau$ of $s$ with $\pi_1(l(s')) = u$.

Now let $\omega = (t_i | i \in \Nbb)$ be any end of $S_{\sigma}$. For each $n$, let $\tau_n$ be the set of those $s \in \tau$ with $\pi_1(l(s)) = t_n$. Then, repeatedly using what we have just shown, it follows by induction on $n$ that each $\tau_n$ is finite. Let $f_n \colon \tau_{n+1}  \to \tau_n$ be given by restriction. Then by K\"onig's Infinity Lemma we can find $s^n \in \sigma$ with $\widehat {s^n} \in \tau_n$ for each $n$ such that $f_n(\widehat {s^{n+1}}) = \widehat {s^n}$ for each $n$. Let $s$ be the infinite sequence $s^1_1s^1_2s^2_3s^2_4s^3_5s^3_6...$. Then $s$ is an infinite play according to $\sigma$ by \autoref{change_a-little}, so since $\sigma$ is winning we have $\omega \in \Psi$. 
\end{proof}

By \autoref{reduced}, we may assume without loss of generality that Sarah's winning strategy $\sigma$ is reduced, and so that all ends of $S_{\sigma}$ are in $\Psi$. It follows, using the argument of \cite[Lemma 8.2]{BC:determinacy}, that there is a $\Psi$-pre-vector $(S, \overline v)$ with $e \in \supp(S, \overline v) \subseteq P+e$. This completes our proof of (O2).

\section{(IM) for trees of presentations}\label{IM}

Our aim in this section is to show that gluing together stellar trees of presentations gives presentations which satisfy (IM), which is the only remaining part of the task of showing that this construction gives rise to matroids. To prove (IM), it suffices by \autoref{base} to show that we can construct a base. We will do this recursively, successively building the parts of the base at each node of the tree. When building the part of the base at a particular node, we will want to ignore the details of the branches of the tree which remain when this node is removed. To this end, we will replace each such branch by a finite matroid which retains just enough information for our argument. This will be done with the following Lemma:

\begin{lem}\label{pickPQ}
Let $\Pi = (V, W)$ be a presentation on a set $E$, and let $F$ be a finite subset of $E$. Then there are disjoint subsets $P_F$ and $Q_F$ of $E \setminus F$ such that $E \setminus (P_F \cup Q_F)$ is finite and $\Pi' \restric_F = \Pi \restric_F$ and $\Pi'.F = \Pi.F$, where $\Pi' = \Pi/P_F \backslash Q_F$.
\end{lem}
\begin{proof}
Let $B_V$ be a (linear) basis of $V.F$ and $B_W$ a (linear) basis of $W.F$. For each $v \in B_V$, choose some $\hat v \in V$ with $\hat v \restric_F = v$. Similarly, for each $w \in B_W$ choose some $\hat w \in W$ with $\hat w \restric_F = w$. Let $F' = F \cup \left[\left(\bigcup_{v \in B_V} \supp(\hat v) \right) \cap \left( \bigcup_{w \in B_W} \supp(\hat w) \right)\right]$, which is finite because it is the union of $F$ with a finite union of sets of the form $\supp(\hat v) \cap \supp(\hat w)$. Let $P_F = \bigcup_{v \in B_V} \supp(\hat v) \setminus F'$, and $Q_F = E \setminus (P_F \cup F')$. Thus $P_F$ and $Q_F$ are disjoint, and $E \setminus (P_F \cup Q_F) = F'$ is finite.

For each $v \in B_V$, we have $\supp(\hat v) \subseteq E \setminus Q_F$, so $v \in (V \setminus Q_F).F$. Thus $V.F \subseteq (V \setminus Q_F).F$. It is clear that the reverse inclusion $(V \setminus Q_F).F \subseteq V.F$ also holds, and so $(V \setminus Q_F).F = V.F$. Since by \autoref{finpres} any presentation on a finite set is determined by its set of vectors, we may deduce that $\Pi'.F = (\Pi \setminus Q_F).F = \Pi.F$. The proof that $\Pi' \restric_F = \Pi \restric_F$ is similar.
\end{proof}

Using this, we can now obtain the lemma which will be applied at each node:

\begin{lem}\label{IMstar}
Let $\Pi$ be a presentation on a set $E$ which is stellar with respect to a set $\Fcal$ of disjoint subsets of $E$. Let $\Tcal = (T, \overline V, \overline W)$ be a tree of presentations, where $T$ is a star with central node $*$ and leaf set $\Fcal$, and $(\overline V(*), \overline W(*)) = \Pi$ and for each $F \in \Fcal$ we have $E(*F) = F$. Let $E' = E \setminus \bigcup \Fcal$. Let $X$ and $Y$ be disjoint subsets of $E(\Tcal)$ such that $X$ is $S(V(\Tcal))$-independent and $Y$ is $S(W(\Tcal))$-independent. Then there are disjoint subsets $X'$ and $Y'$ of $E(\Tcal)$ extending $X$ and $Y$ respectively such that:

\begin{itemize}
\item $E' \subseteq X' \cup Y'$
\item $X'$ is $S(V(\Tcal))$-independent and $Y'$ is $S(W(\Tcal))$-independent.
\item For any $e \in E' \setminus X'$ there is some $C \in S(V(\Tcal))$ with $e \in C \subseteq X' + e$.
\item For any $e \in E' \setminus Y'$ there is some $D \in S(W(\Tcal))$ with $e \in D \subseteq Y' + e$.
\item There do not exist leaves $F, F' \in \Fcal$ such that there is a connected component of $\Pi(\Tcal)/X' \backslash Y'$ meeting both $E(F)$ and $E(F')$. 
\end{itemize}
\end{lem}
\begin{proof}
For each $F \in \Fcal$ we will denote the presentation $(\overline V(F), \overline W(F))$ by $\Pi_F$.
By \autoref{minors} we may assume without loss of generality that $X$ and $Y$ are subsets of $E'$.

We begin by picking, for each $F \in \Fcal$, sets $P_F$ and $Q_F$ as in \autoref{pickPQ} for the finite subset $F$ of $E(F)$. Let $\tilde V(F) = \overline V(F)/P_F\backslash Q_F$ and $\tilde W(F) = \overline W(F) \backslash P_F / Q_F$. Taking $\tilde V(*) = V_{\Pi}$ and $\tilde W(*) = W_{\Pi}$ we get an $\Fcal$-star $\tilde \Tcal = (T, \tilde V, \tilde W)$ of presentations around $\Pi$. By construction, $X$ is $S(V(\tilde \Tcal))$-independent and $Y$ is $S(W(\tilde \Tcal))$-independent. Since $\Pi$ is stellar, we can choose a base $B$ extending $X$ and disjoint from $Y$ for the matroid $M$ presented by $\Pi(\tilde \Tcal)$: let $B'$ be the base of the dual matroid $M^*$ given by taking the complement of $B$. For each $F \in \Fcal$, let $X_F$ be an independent subset of $P_F \cup (B \cap E(F))$ such that $(\Pi_F/X_F)\restric_F = (\Pi_F/(P_F \cup (B \cap E(F))))\restric_F$ as in \autoref{fromO3} and let $Y_F$ be a coindependent subset of $Q_F \cup (B' \cap E(F))$ such that $(\Pi_F\backslash Y_F).F = (\Pi_F\
backslash (Q_F \
\cup (B' \cap E(F)))).F$. Note that $(\Pi_F \backslash Y_F).F = (\Pi_F / X_F) \restric_F$. Let $X' = (B \cap E') \cup \bigcup_{F \in \Fcal}X_F$ and $Y' = (B' \cap E') \cup \bigcup_{F \in \Fcal}Y_F$. It is clear that $X'$ and $Y'$ are disjoint, cover $E'$, and respectively extend $X$ and $Y$. 

Now suppose for a contradiction that $X'$ is $S(V(\Tcal))$-dependent. Then there is some $\Tcal$-prevector $(S, \hat v)$ whose support $C$ is nonempty and included in $X'$. The tree $S$ cannot consist of just a single leaf of $T$ by independence of of the sets $X_F$, so it must contain $*$. For each leaf $F$ of $T$ in $S$, we have $\hat v(*) \restric_F \in \overline V(F)/X_F$, so by the definition of $X_F$ we have $\hat v(*) \in (\overline V(F)/(P_F \cup (B \cap E(F))))\restric_F$, that is, there is some vector $\hat v'(F)$ of $\tilde V(F)$ whose support is included in $(B \cap E(F)) \cup F$ and with $\hat v'(F) \restric_F = \hat v(*) \restric_F$. Letting $\hat v'(*) = \hat v(*)$, we obtain a $\tilde \Tcal$-prevector $(S, \hat v')$ whose support is included in $B$, and so must be empty. So for each leaf $F$ of $T$ in $S$ we have $\hat v'(F) \in k^F$ and so, by our choice of $P_F$ and $Q_F$, $\hat v'(F) \in \overline V(F)$, so since $\hat v'(F) = \hat v(*) \restric_F$ we have 
$\hat v(*) \restric_F \in \overline V(F)$. Also, 
$C$ cannot meet $E'$, so since $C$ is nonempty there is some leaf $F$ of $T$ in $S$ for which the support of $\hat v(F)$ isn't a subset of $F$. Then $\hat v(F) - \hat v(*) \restric_{F}$ is a vector in $\overline V(F)$ whose support is nonempty and included in $X_F$, contradicting the independence of $X_F$.

This shows that $X'$ is $S(V(\Tcal))$-independent, and a dual argument shows that $Y'$ is $S(W(\Tcal))$-independent.

Next we will show that for any $e \in E' \setminus X'$ there is some $C \in S(V(\Tcal))$ with $e \in C \subseteq X' + e$. Since $e \in B'$, there is some circuit $C_0$ of $M$ with $e \in C_0 \subseteq B + e$. Let $(S, \hat v)$ be a $\tilde \Tcal$-prevector with support $C_0$. Then for each leaf $F$ of $T$ in $S$, we have $\hat v(*) \restric_F \in (\overline V(F)/(P_F \cup (B \cap E(F))))\restric_F = (\overline V(F)/X_F)\restric_F$, so that there is some $\hat v'(F) \in \overline V(F)$ with $\supp(\hat v'(F)) \subseteq X_F \cup F$ and $\hat v'(F) \restric_F = \hat v(*) \restric_F$. Letting $\hat v '(*) = \hat v(*)$, we get a $\Tcal$-prevector $(S, \hat v')$ whose support is the desired $C$. A dual argument shows that for any $e \in E' \setminus Y'$ there is some $D \in S(W(\Tcal))$ with $e \in D \subseteq Y' + e$.

It remains to prove the final condition of the Lemma. Suppose for a contradiction that this condition fails, and let $F\in \Fcal$ such that there is a connected component of $\Pi(\Tcal)/X' \backslash Y'$ containing some edge $e$ of $E(F)$  and some edge $e'$ of $E(F')$ for some $F' \neq F$. Let $C$ be a minimal nonempty element of $S(V(\Tcal)/X' \backslash Y')$ containing both $e$ and $e'$, and let $(S, \hat v)$ be a $\Tcal$-prevector whose support includes $C$ but is a subset of $C \cup X'$. Both $F$ and $*$ must be in $S$. Then the support of $\hat v(F)$ can't meet $Y_F$, so $\hat v(F) \restric_F$ is a vector of $(\overline V(F) \backslash Y_F).F = (\overline V(F) / X_F) \restric_F$, so that there is some $v \in \overline V(F)$ with $\supp(v) \subseteq X_F \cup F$ and $v \restric_F = \hat v(F) \restric_F$. Then $(\{F\}, F \mapsto v(F) - v)$ is a $\Tcal$-prevector whose support is a subset of $C \cup X'$ containing $e$ but not $e'$, contradicting the minimality of $C$. This completes the proof.
\end{proof}

We now apply this lemma recursively to build the necessary bases. We will need a little notation for our recursive construction. For any tree $T$ and directed edge $st$ of $T$, let $T_{s \to t}$ be the subtree of $T$ on the set of vertices $u$ for which the unique path from $s$ to $u$ in $T$ contains $t$. For $\Tcal = (T, \bar V, \bar W)$ a tree of presentations and $st$ a directed edge of $G$, let $\Tcal_{s \to t}$ be the tree of presentations $(T_{s \to t}, \bar V\restric_{T_{s \to t}}, \bar W \restric_{T_{s \to t}})$.

\begin{thm}
Let $\Tcal = (T, \overline V, \overline W)$ be a stellar tree of presentations, and let $\Psi$ be a Borel set of ends of $\Tcal$. Then $\Pi_{\Psi}(\Tcal)$ presents a matroid.
\end{thm}
\begin{proof}
We have already shown that $\Pi_{\Psi}(\Tcal)$ is a presentation. Indeed, our results so far show that for each edge $tt'$ of $\Tcal$ the pair $\Pi_{\Psi}(\Tcal_{t \to t'})$ is a presentation.

It remains to show that $S(V_{\Psi}(\Tcal))$ satisfies (SM), for which by \autoref{minors} and \autoref{base} it is enough to show that there is some partition of $E(\Tcal)$ into a base $X$ and a cobase $Y$, that is, $Y$ is a subset of the $S(V_{\Psi}(\Tcal))$-span of $X$ and $X$ is a subset of the $S(W_{\Psi})(\Tcal))$-span of $Y$. We build $X$ and $Y$ recursively. More precisely, we pick a root $t_0$ for $T$ and order the vertices of $T$ by the tree order $\leq$ with respect to this root. This is a well-founded order, and we construct subsets $X_t$ and $Y_t$ of $E(\Tcal)$ for each node $t$ of $T$ by recursion over $\leq$ such that:

\begin{enumerate}
\item $X_t$ and $Y_t$ are disjoint.
\item $X_t \subseteq X_{t'}$ and $Y_t \subseteq Y_{t'}$ for $t \leq t'$.
\item $X_{t'} \setminus X_{t} \subseteq E(\Tcal_{t \to t'})$ and  $Y_{t'} \setminus Y_{t} \subseteq E(\Tcal_{t \to t'})$ for any edge $tt'$ of $T$ with $t \leq t'$.
\item $E(t) \cap E(\Tcal) \subseteq X_t \cup Y_t$.
\item $X_t$ is $S(V_{\Psi}(\Tcal))$-independent and $Y_t$ is $S(W_{\Psi}(\Tcal))$-independent.
\item For any $e \in E(t) \cap E(\Tcal) \setminus X_t$ there is some $C \in S(V(\Tcal))$ with $e \in C \subseteq X_t + e$
\item For any $e \in E(t) \cap E(\Tcal) \setminus Y_t$ there is some $D \in S(W(\Tcal))$ with $e \in D \subseteq Y_t + e$
\item There is no edge $tt'$ of $T$ with $t \leq t'$ such that there is a connected component of $\Pi(\Tcal)/X_t \backslash Y_t$ meeting both $E(\Tcal_{t \to t'})$ and $E(\Tcal_{t' \to t})$.
\end{enumerate}
If we can find such $X_t$ and $Y_t$ then the sets $X = \bigcup_{t \in V(T)} X_t$ and $Y = \bigcup_{t \in V(T)}(Y_t)$ will give the base and cobase we require: they will be disjoint by conditions 1, 2 and 3, will cover by condition 4 and will be respectively spanning and cospanning by conditions 6 and 7. It remains to show that this recursive construction can be carried out.

We construct $X_{t_0}$ and $Y_{t_0}$ by applying \autoref{IMstar} to the star of presentations with central node $\Pi(t_0)$ and with a leaf for each neighbour $t$ of $t_0$ in $T$ labelled with the presentation $\Pi_{\Psi}(\Tcal_{t_0 \to t})$, and taking $X = Y = \emptyset$.

The construction of $X_t$ and $Y_t$ for $t \neq t_0$ is very similar. Let $s$ be the predecessor of $t$ in the tree order. Let $E'$ be the set $E(t) \cap E(\Tcal)$ of real edges of $E(t)$. Let $T'$ be the subtree $T_{t\to s}+t$ of $T$. Let $\Pi = (\Pi_{\Psi}(\Tcal') / (X_s \cap E(\Tcal_{t \to s}))) \restric_{E(t) \setminus E(st)}$. Note that by condition 8 applied at $s$ we also have $\Pi = (\Pi_{\Psi}(\Tcal') \backslash (Y_s \cap E(\Tcal_{t \to s}))).(E(t) \setminus E(st))$. Then we build $X'$ and $Y'$ by applying \autoref{IMstar} to the star of presentations with central node $\Pi$ and with a leaf for each successor $t'$ of $t$ in $T$ labelled with the presentation $\Pi_{\Psi}(\Tcal_{t \to t'})$. We take the $X$ and $Y$ of the Lemma to be the intersections of $X_s$ and $Y_s$ with $E(\Tcal_{s \to t})$ respectively. Finally, we let $X_t = X_s \cup X'$ and $Y_t = Y_s \cup Y'$.
\end{proof}

\bibliographystyle{plain}
\bibliography{literatur}
\end{document}